\documentclass[reqno,12pt]{amsart} 
\newcommand{\supp}{\mathop{\mathrm{supp}}\nolimits}
\numberwithin{equation}{section}
\usepackage{amssymb}
\usepackage[mathscr]{eucal}
\usepackage{amsmath}
\usepackage{latexsym}
\usepackage{amsthm}
\theoremstyle{plain} 
\newtheorem{theorem}{\indent\sc Theorem}[section]
\newtheorem{lemma}[theorem]{\indent\sc Lemma}

\newtheorem{proposition}[theorem]{\indent\sc Proposition}

\theoremstyle{definition} 

\newtheorem{remark}[theorem]{\indent\sc Remark}

\begin{document}
\setcounter{page}{1}
\title[Weighted Weak (1,1) estimates for one-sided oscillatory singular integrals]
{Weighted Weak (1,1) estimates for one-sided oscillatory singular integrals$^*$}
\author[Zunwei Fu ]{Zunwei Fu}
\author[Shanzhen Lu]{Shanzhen Lu} 
\author[Shuichi Sato]{Shuichi Sato} 
\author[Shaoguang Shi]{Shaoguang Shi} 
\thanks{
\endgraf
}

%
\thanks{ 
$^{*}$This work was partially supported by
 NSF of China (Grant Nos. 10871024, 10901076 and 10931001), NSF of Shandong Province
 (Grant No. Q2008A01) and the Key Laboratory of Mathematics and Complex System (Beijing Normal
University), Ministry of Education, China.}

\address{
School of Sciences\endgraf
Linyi
University \endgraf
Linyi 276005\endgraf
P. R. China}
\email{lyfzw@tom.com}

\address{
School of Mathematical Sciences\endgraf
Beijing Normal
University\endgraf
Beijing 100875\endgraf
P. R. China}
\email{lusz@bnu.edu.cn}

\address{
Department of Mathematics \endgraf
Faculty of Education \endgraf
Kanazawa
University\endgraf
Kanazawa 920-1192\endgraf
Japan}
\email{shuichi@kenroku.kanazawa-u.ac.jp}
\address{
School of Mathematical Sciences\endgraf
Beijing Normal
University\endgraf
Beijing 100875\endgraf
and\endgraf
School of Sciences\endgraf
Linyi
University \endgraf
Linyi 276005\endgraf
P. R. China}
\email{shishaoguang@yahoo.com.cn}


\maketitle
\begin{abstract}
We consider one-sided weight classes of Muckenhoupt type and
study the weighted weak type (1,1) norm inequalities of a class of one-sided oscillatory singular integrals with smooth kernel.
\end{abstract}
\section{Introduction} 
 Oscillatory integrals have been an essential part of harmonic analysis; three 
chapters are devoted to them in the celebrated Stein's book \cite{St}.
  Many important operators in harmonic analysis are some versions of oscillatory integrals, such as the Fourier transform, the Bochner-Riesz means, the Radon 
transform in CT technology and so on. For a more complete account on  
oscillatory integrals in classical harmonic analysis, we would like to refer the interested reader to \cite{G}, \cite{Lu1}, \cite{Lu2}, \cite{LDY}, \cite{LZ}, 
\cite{PS} and references therein. Another early impetus for the study of 
oscillatory integrals came with their application to number theory \cite{B}. 
In more recent times, the operators fashioned from oscillatory integrals, such 
as pseudo-differential operator in PDE become another motivation to study 
them. Based on the estimates of some kinds of oscillatory integrals, one can 
establish the well-posedness theory of a class of dispersive equations, 
for some of this works, we refer to \cite{CM}, \cite{KPV1}, \cite{KPV2}.

This paper is focused on a class of oscillatory singular integrals related to 
the one defined by Ricci and Stein \cite{RS}
$$Tf(x)=\mathrm{p.v.}\int_{\mathbb{R}}e^{iP(x,y)}K(x-y)f(y)\,dy,$$
where $P(x,y)$ is a real valued polynomial defined on $\mathbb{R}\times\mathbb{R}$, and $K\in C^{1}(\mathbb{R}\setminus\{0\})$ is a Calder\'{o}n-Zygmund kernel which satisfies:
\begin{gather}
|K(x)|\leq\frac{C}{|x|},\ \ \ \    |\nabla K(x)|\leq \frac{C}{|x|^{2}},
\\
 \int_{a<|x|<b}K(x)\,dx =0 \qquad \text{for all $a, b$ $(0<a<b)$.}
 \end{gather}

\begin{theorem}[\cite{RS}]    Suppose $K$ satisfies $(1.1)$, $(1.2)$.
Then for any real polynomial $P(x,y)$, the oscillatory singular integral
operator $T$
is of type $(L^{p}(\mathbb{R}), L^{p}(\mathbb{R}))$, $1<p<\infty$, where its
operator norm is bounded by a constant depending on the total degree of
$P$, but not on the coefficients of $P$ in other respects.
\end{theorem}

Let $A_{p}(1<p<\infty)$ denote the Muckenhoupt classes \cite{CF}. This class consists of positive locally integrable functions (weight functions) $w$ for which
$$\sup_I\left(\frac{1}{|I|}\int_{I}w(x)dx\right)\left(\frac{1}{|I|}\int_{I}w(x)^{1-p'}dx\right)^{p-1}<\infty,$$
where the supremum is taken over all intervals $I\subset\mathbb{R}$ and $1/p+1/p'=1$. 

In 1992, Lu and Zhang \cite{LZ} gave the weighted result of Theorem 1.1.
\begin{theorem} Suppose $K$ satisfies $(1.1)$, $(1.2)$.
Then for any real polynomial $P(x,y)$, the oscillatory singular integral
operator $T$ is of type $(L^{p}(w),L^{p}(w))$, where $w\in A_{p}$, $1<p<\infty$. Here its
operator norm is bounded by a constant depending on the total degree of
$P$, but not on the coefficients of $P$ in other respects.
\end{theorem}

For the case $p=1$, Chanillo and Christ \cite{CC} gave a supplement for Theorem 1.1.
\begin{theorem}
Under the same assumption as in Theorem $1.1$, we have
$$
||Tf||_{L^{1,\infty}}\leq C\|f\|_{L^{1}},
$$
where $L^{1,\infty}$ denotes the weak $L^{1}$ space, and the constant $C$
is independent of $P$ if the total degree of the polynomial is fixed.
\end{theorem}

Let $ A_{1}$ be the class of weight functions $w$ satisfying
$Mw(x)\leq Cw(x)$  a.e., where $M$ denotes
the Hardy-Littlewood maximal operator
$$Mf(x)=\sup_{h>0}\frac{1}{2h}\int_{x-h}^{x+h}|f(y)|\,dy. $$
We write $w(E)=\int_E w$  for a measurable set $E$. The third author of this paper gave the weighted version of
Theorem 1.2.
\begin{theorem}[\cite{Sa}]
Under the same assumption as in Theorem $1.1$, if $w\in A_{1}$, then
$$
\sup_{\lambda>0}\lambda w\left(\left\{x\in \mathbb{R}:|Tf(x)|>\lambda\right\}\right)\leq C\|f\|_{L^{1}(w)}
$$
where $C$ depends on the total degree of $P$ and, in other respects,  is
independent of the coefficients of $P$.
\end{theorem}

The study of weights for one-sided operators was motivated not only
as the generalization of the theory of both-sided ones but also their
natural appearance in harmonic analysis; for example, it is required when
we treat the one-sided Hardy-Littlewood maximal operator \cite{Saw}
$$M^{+}f(x)=\sup_{h>0}\frac{1}{h}\int_{x}^{x+h}|f(y)|\,dy,\eqno(1.3)$$
and
$$M^{-}f(x)=\sup_{h>0}\frac{1}{h}\int_{x-h}^{x}|f(y)|\,dy\eqno(1.4)$$
arising in the ergodic  maximal function. The classical Dunford-schwartz
ergodic theorem can be considered as the first result about weights for (1.3) and (1.4). In \cite{Saw}, Sawyer introduced the one-sided $A_{p}$ classes
$A_{p}^{+}$, $A_{p}^{-}$; they are defined by the following conditions:
$$
A_{p}^{+}:\quad A_{p}^{+}(w):=\sup_{a<b<c}\frac{1}{(c-a)^{p}}\int_{a}^{b}w(x)\,dx\left(\int_{b}^{c}w(x)^{1-p'}\,dx\right)^{p-1}<\infty,
$$
$$
A_{p}^{-}:\quad A_{p}^{-}(w):=\sup_{a<b<c}\frac{1}{(c-a)^{p}}\int_{b}^{c}w(x)\,dx\left(\int_{a}^{b}w(x)^{1-p'}\,dx\right)^{p-1}<\infty,
$$
 when $1<p<\infty$;  also, for $p=1$,
$$
A_{1}^{+}:\quad M^{-}w\leq Cw,
$$
$$
A_{1}^{-}:\quad M^{+}w\leq Cw,
$$
for some constant $C$.
The smallest constant $C$ for which the above inequalities are satisfied will be denoted by $A_{1}^{+}(w)$ and $A_{1}^{-}(w)$.
$A_{p}^{+}(w)$ ($A_{p}^{-}(w)$), $p\geq 1$, will be called the
$A_{p}^{+}$ ($A_p^-$) constant of $w$.

\begin{theorem}[\cite{Saw}]  Let $M^+$ be as in $(1.3)$.
\begin{enumerate}
\renewcommand{\labelenumi}{(\alph{enumi})}
\item Let $1\leq p<\infty$. Then there exists $C>0$ such that the inequality
$$\sup_{\lambda>0} \lambda^pw\left(\left\{x\in \mathbb{R}:|M^{+}f(x)|>\lambda\right\}\right)\leq C\|f\|_{L^{p}(w)}^p$$
holds for all $f$, if and only if $w\in A_{p}^{+}$.
\item Let $1<p<\infty$. Then there exists $C>0$ such that the inequality
$$\|M^{+}f\|_{L^{p}(w)}\leq C\|f\|_{L^{p}(w)}$$
holds for all $f\in L^{p}(w)$, if and only if $w\in A_{p}^{+}$.
\end{enumerate}
\end{theorem}
\begin{remark}
Let us remark here and after that similar results can be obtained for
the left-hand-side operator
by changing the condition  $A_{p}^{+}$ by $A_{p}^{-}$.
\end{remark}
Together with the characterizations of the weighted inequalities for $M^{+}$ and $M^{-}$, Sawyer obtained some properties of the classes $A_{p}^{+}$ and $A_{p}^{-}$.
\begin{proposition}[\cite{Saw}]\begin{enumerate}
\renewcommand{\labelenumi}{(\alph{enumi})}
\item If $w\in A_{1}^{+}$, then $w^{1+\varepsilon}\in A_{1}^{+}$ for some
$\varepsilon>0$.
\item $w\in A_{p}^{+}$ for $1<p<\infty$, if and only if there exists $w_{1}\in A_{1}^{+}$ and $w_{2}\in A_{1}^{-}$ such that $w=w_{1}(w_{2})^{1-p}$.
\item If $1\leq p<\infty$, then $A_{p}=A_{p}^{+}\bigcap A_{p}^{-}$, $A_{p}\subset A_{p}^{+}$, $A_{p}\subset A_{p}^{-}$.
\item   $A_p^+\subset A_r^+$, $A_p^-\subset A_r^-$ if $1\leq p\leq r$.
\end{enumerate}
\end{proposition}
Perhaps it is worth pointing out that these classes not only control the boundedness of $M^{+}(M^{-})$, but also they are the right weight classes for one-sided singular integrals \cite{AFM}, and they also appear in PDE \cite{GS}.

We say a Calder\'{o}n-Zygmund kernel $K$ is a one-sided Calder\'{o}n-Zygmund kernel (OCZK) if $K$ satisfies $(1.1)$ and
$$\left| \int_{a<|x|<b}K(x)\, dx\right|\leq C, \quad 0<a<b\eqno(1.5)
$$
with support in $\mathbb{R}^{-}=(-\infty,0)$
or $\mathbb{R}^{+}=(0,+\infty)$.
The smallest constant for which (1.1) and (1.5) hold will be denoted by $C(K)$. In \cite{AFM}, Aimar, Forzani and Mart\'{\i}n-Reyes studied the one-sided Calder\'{o}n-
Zygmund singular integrals which are defined by
$$
\widetilde{T}^{+}f(x)=\lim_{\varepsilon\rightarrow0^{+}}\int_{x+\varepsilon}^{\infty}K(x-y)f(y)\,dy
$$
and
$$
\widetilde{T}^{-}f(x)=\lim_{\varepsilon\rightarrow0^{+}}\int_{-\infty}^{x-\varepsilon}K(x-y)f(y)\,dy
$$
where the kernels $K$ are OCZKs. 

\begin{theorem}[\cite{AFM}]  Let $K$ be a OCZK with
support in $\mathbb{R}^{-}=(-\infty,0)$. Then
\begin{enumerate}
\renewcommand{\labelenumi}{(\alph{enumi})}
\item $\widetilde{T}^{+}$ is bounded on $L^{p}(w) (1<p<\infty)$ if
$w\in A_{p}^{+}$.
\item  $\widetilde{T}^{+}$ maps $L^{1}(w)$ into $ L^{1,\infty}(w)$
if $w\in A_{1}^{+}$.
\end{enumerate}
\end{theorem}
Also, a result concerning the converse of Theorem 1.8 is given in \cite{AFM}.
Inspired by \cite{CC}, \cite{Sa} and \cite{Saw}, we will study the one-sided version of
Theorem 1.4 by the aid of induction, Calder\'{o}n-Zygmund decomposition,
estimates for oscillatory integrals of the unweighted case and interpolation of operators with change of measures. In the foregoing and following, the letter
$C$ will stands for a positive constant which may vary from line to line.


\section{Main Results}
We first give the definition of one-sided oscillatory singular integral
operators $T^{+}, T^{-}$:
$$\begin{array}{rl}
\displaystyle T^{+}f(x)&=\displaystyle\lim_{\varepsilon\rightarrow0^{+}}\int_{x+\varepsilon}^{\infty}e^{iP(x,y)}K(x-y)f(y)\,dy
\\[4mm]&=\displaystyle \mathrm{p.v.}\int_{x}^{\infty}e^{iP(x,y)}K(x-y)f(y)\,dy,
\end{array}$$
and
$$\begin{array}{rl}
\displaystyle
T^{-}f(x)&=\displaystyle\lim_{\varepsilon\rightarrow0^{+}}\int_{-\infty}^{x-\varepsilon}e^{iP(x,y)}K(x-y)f(y)\,dy
\\[4mm]&=\displaystyle \mathrm{p.v.}\int_{-\infty}^{x}e^{iP(x,y)}K(x-y)f(y)\,dy,
\end{array}$$
where $P(x,y)$ is a real polynomial defined on $\mathbb{R}\times\mathbb{R}$, and the kernels $K$ are  OCZKs with support in
$\mathbb{R}^{-}$ and $\mathbb{R}^{+}$, respectively.

Now, we  formulate our result as follows:
\begin{theorem}
If $w\in A_{1}^{+}$, then there exists a constant C depending on the total
degree of $P$, $C(K)$ and $A_{1}^{+}(w)$ such that
$$
\sup_{\lambda>0}\lambda w(\{x\in \mathbb{R}:|T^{+}f(x)|>\lambda\})\leq
C\|f\|_{L^{1}(w)},     \eqno(2.1)
$$
for $f\in \mathscr S(\Bbb R)$ $($the Schwartz class$)$.
\end{theorem}
We shall carry out the proof of Theorem 2.1 by induction, as in \cite{LZ},
\cite{RS} and \cite{Sa}.
Suppose $P(x,y)$ is a real polynomial in $x$ and $y$. First, we assume that 
Theorem 2.1 is valid for all polynomials which are the sums of monomials of 
degree less than $k$ in $x$ and of any degree in $y$,
together with the sums of monomials which are of degree $k$ in $x$ and of 
degree less than $l$ in $y$. Let 
$$
P(x,y)=a_{kl}x^{k}y^{l}+R(x,y),
$$
with
$$
R(x,y)=\sum_{\alpha<k,\beta}a_{\alpha\beta}x^{\alpha}y^{\beta}+
\sum_{\beta< l}a_{k\beta}x^{k}y^{\beta}   
$$
satisfying the above induction assumption.
\par 
Let us now prove that (2.1) holds for $P(x,y)$. 
Arguing as in \cite[p. 188]{RS}, by the aid of weighted theory of one-sided 
Calder\'{o}n-Zygumund operators, 
without loss of generality, we may assume $k>0,l>0$ and $|a_{kl}|\neq 0$ (for 
if $|a_{kl}|= 0$, (2.1) holds by the induction assumption). By dilation 
invariance of the operators and weights, we only need to consider the case 
$|a_{kl}|=1.$
\par
We split the kernel $K$ as
$$
K(x-y)=K(x-y)\chi_{\{|x-y|\leq1\}}(y)+K(x-y)\chi_{\{|x-y|>1\}}(y)=K_{0}(x-y)+K_{\infty}(x-y),
$$ 
where $\chi_E$ denotes the characteristic function of a set $E$,  
and consider the corresponding splitting $T^{+}=T^{+}_{0}+T^{+}_{\infty}$:\\
\begin{gather*}
T^{+}_{0}f(x)=\mathrm{p.v.}\int_{x}^{\infty}e^{iP(x,y)}K_{0}(x-y)f(y)\,dy,
\\ 
T^{+}_{\infty}f(x)=\int_{x}^{\infty}e^{iP(x,y)}K_{\infty}(x-y)f(y)\,dy.
\end{gather*}  

In Section 4, we will prove the following proposition under the induction 
assumption.
\begin{proposition}
If $w\in A_{1}^{+}$, then there exists a constant $C$ depending on the total
degree of $P$, $C(K)$ and $A_{1}^{+}(w)$ such that
$$
\sup_{\lambda>0}\lambda w(\{x\in \mathbb{R}:|T^{+}_{0}f(x)|>\lambda\})\leq C\|f\|_{L^{1}(w)} \eqno(2.2)
$$
and$$
\sup_{\lambda>0}\lambda w(\{x\in \mathbb{R}:|T^{+}_{\infty}f(x)|>\lambda\})
\leq C\|f\|_{L^{1}(w)}.\eqno(2.3)
$$
\end{proposition}
Obviously,  this will complete the proof of Theorem 2.1.

The rest of this paper is devoted to the argument for Proposition 2.2. Section 3 contains some preliminaries which are essential to our
proof. In Section 4, we prove Proposition 2.2;   this part is partially
motivated by \cite{LZ} and \cite{Sa}.

\section{Preliminaries}

Let $w\in A_1^+$, $f\in \mathscr S(\Bbb R)$.  We perform the following
Calder\'{o}n-Zygmund decomposition at height $\lambda>0$.
\begin{lemma}\label{Lemma}
We have a collection $\{I\}$ of non-overlapping closed intervals in $\Bbb R$
 and functions $g, b$ on $\Bbb R$  such that
\begin{gather}
f=g+b,
\\
\lambda\leq |I|^{-1}\int_{I}|f|\leq C\lambda,
\\
w\left(\bigcup I\right)\leq C\lambda^{-1}\|f\|_{L^{1}(w)},
\\
\|g\|_{L^{1}(w)}\leq C\|f\|_{L^{1}(w)},
\\
\|g\|_{\infty}\leq C\lambda,
\\
b = \sum_I b_I, \quad
\supp (b_I) \subset I,
\quad
\int b_I = 0,
\quad
\|b_I\|_{L^1} \leq C\lambda |I|.
\end{gather}
\end{lemma}
\begin{proof}
Let
$$\left\{x\in \Bbb R: M^+f(x)>\lambda\right\}=\bigcup I'$$
be the component decomposition. Let $I$ be the closure of $I'$.
By Lemma 2.1 of \cite{Saw} we see that $|I|^{-1}\int_{I} |f|\geq\lambda$, which proves (3.2).
Define $b_I=\left(f-|I|^{-1}\int_{I}f\right)\chi_I$, $b = \sum_I b_I$
 and $g=f\chi_F +\sum_I |I|^{-1}\left(\int_I f\right)\chi_I$, where
 $F=\mathbb{R}\setminus \bigcup I$.
Then,
we only need to prove (3.3) and (3.4) because (3.1), (3.5) and (3.6) are
 straightforward.
 \par
Let $I$ be one of the intervals obtained above.
By Lemma 1 of  \cite{MOT} and Lemma 2.1 of \cite{Saw}, for any positive
increasing function $U_I$ on $I$ we have
\begin{equation}
\int_I U_I\leq \lambda^{-1}\int_I U_I|f|.
\end{equation}
Also, since $w\in A_1^+$, by Lemma 2 of \cite{MOT}, there exists a positive
increasing function
$V_{w,I}$ on $I$ such that
\begin{equation}
V_{w,I}\leq Cw \text{\,\, a.e. on $I$,} \qquad \int_I w \leq \int_I V_{w,I},
\end{equation}
where $C$ is independent of $I$.
By (3.8) and (3.7) with $V_{w,I}$ in place of $U_I$, we can prove (3.3)
as follows (see \cite[p. 520]{MOT}):
\begin{align*}
w\left(\bigcup I\right)&\leq \sum\int_I w\leq \sum\int_I V_{w,I}
\\
&\leq
\lambda^{-1}\sum\int_I V_{w,I}|f| \leq C\lambda^{-1}\sum\int_I |f|w\leq
C\lambda^{-1}\|f\|_{L^1(w)}.
\end{align*}
The estimate (3.4) can be proved similarly:
\begin{align*}
\|g\|_{L^1(w)}&\leq \int_F|f|w +\sum |I|^{-1}\left|\int_{I}f\right|\int_I w
\\
&\leq \int_F|f|w +C\lambda\sum\int_I V_{w,I}
\\
&\leq \int_F|f|w +C\sum\int_I V_{w,I}|f|
\\
&\leq \int_F|f|w +C\sum\int_I |f|w
\\
&\leq C\|f\|_{L^1(w)}.
\end{align*}
This completes the proof.
\end{proof}
We decompose $K_{\infty}(x,y)=e^{iP(x,y)}K_{\infty}(x-y)=\sum_{j=0}^{\infty}K_{j}(x,y)$, where
$$
K_{j}(x,y)=\varphi(2^{-j}(x-y))K_{\infty}(x,y),
$$
and $\varphi\in C_{0}^{\infty}(\mathbb{R})$ such that
$\supp(\varphi)\subset \{1/2\leq|x|\leq2\}$, $\sum_{j=0}^{\infty}
\varphi(2^{-j}x)=1$ if $|x|\geq1$.
For $j\geq 0$, we define
\begin{equation}
W_{j}^{+}(f)(x)=\int K_{j}(x,y)f(y)\,dy.
\end{equation}
Let
$$
W^{+}(f)(x)=\sum_{j=1}^{\infty}W_{j}^{+}(f)(x).
$$
Then $T_\infty=W_0^+ +W^+$.   We set
$$\mathcal{B}_{i}=\sum_{2^{i-1}<|I|\leq 2^{i}}b_{I}
\quad (i\geq 1),\,\,\,\,\,\,
\mathcal{B}_{0}=\sum_{|I|\leq1}b_{I}$$
and put $\mathcal{E}=\bigcup \tilde{I}$, where $\tilde{I}$ denotes the interval with the same right end point as $I$ and with length 100 times that of $I$.
When $x\in \mathbb{R}\setminus \mathcal{E}$, we have
\begin{eqnarray*}
W^{+}(b)(x)
&=& W^{+}\left(\sum_{i\geq 0}\mathcal{B}_{i}\right)(x)\\
&=& \sum_{i\geq 0}\sum_{j\geq 1}\int K_{j}(x,y)\mathcal{B}_{i}(y)\, dy\\
&=& \sum_{s\geq 1}\sum_{j\geq s} W^{+}_{j}(\mathcal{B}_{j-s})(x).
\end{eqnarray*}
\begin{lemma}
Suppose that $w\in A_{1}^{+}$ and $s$ is a positive integer.
For $\alpha>0$,  put
$$E_\alpha^s=\left\{x\in \Bbb R  : \left|\sum_{j\geq s}
W_j^+(\mathcal B_{j-s})(x)\right|>\alpha\right\}.$$
Then, there exists $\varepsilon>0$ such that
$$w(E_\lambda^s)\leq C\lambda^{-1} 2^{-\varepsilon s}
\int |f(x)|w(x) \,dx. $$
\end{lemma}
 Lemma 3.2 will be proved by applying a variant of interpolation
 argument of \cite{V} (see \cite{FS1, FS2}). We first give some lemmas which
 are essential to our analysis. Some of them are almost the same as their
 appearances in \cite{CC}, \cite{FS1}, \cite{FS2} and \cite{Sa}.
 Our results differ from the previous ones only in that
we set up them based on one-sided singular integrals and the weight
$w\in A_{1}^{+}$. We use some results and notations given in \cite{Sa}.
Let $\lambda>0$ and $\{\mathcal{G}_{j}\}_{j\geq 0}$ be a family of measurable functions such that
$$
\int_{I}|\mathcal{G}_{j}|\leq \lambda |I|
$$
for all intervals $I$ in $\mathbb{R}$ with length $|I|=2^{j}$.
\begin{lemma}[\cite{Sa}]
Suppose $\sum_{j\geq 0}\|\mathcal{G}_{j}\|_{L^{1}}<\infty$. Then,
for any positive integers $s$, we have
$$
\left\|\sum_{j\geq s} W^{+}_{j}(\mathcal{G}_{j-s})\right\|_{L^{2}}^{2}\leq C\lambda2^{-s}\sum_{j\geq 0} \| \mathcal{G}_{j}\|_{L^{1}}.
$$
\end{lemma}
For each $j\geq0$, let $\mathcal{I}_{j}$ be a family of non-overlapping closed
 intervals $I$ such that $|I|\leq2^{j}$. We assume $I$ and $J$ are non-overlapping if $I\in \mathcal{I}_{i}$, $J\in \mathcal{I}_{j}$ for $i\neq j$ and $\sum_{j\geq0}\sum_{I\in \mathcal{I}_{j}}|I|<\infty$. Put $\mathcal{I}=\bigcup_{j\geq 0}\mathcal{I}_{j}$. Let $\lambda>0$.  For each $I\in \mathcal{I}$, we associate
 $f_{I}\in L^{1}$ such that
$\int |f_{I}|\leq \lambda |I|, \supp(f_{I})\subset I.$
Define
$$
\mathcal{F}_{i}=\sum_{I\in \mathcal{I}_{i}}f_{I}.
$$
\begin{lemma}
Let $w\in A_{1}^{+}$ and $s$ be a positive integer. Then
$$
\left\|\sum_{j\geq s} W^{+}_{j}(\mathcal{F}_{j-s})\right\|_{L^{1}(w)}
\leq C_{w}\lambda\sum_{J\in \mathcal{I}}|J|\inf_{J}w,
$$
where $\inf_{J}f=\inf_{x\in J}f(x)$.
\end{lemma}
\begin{proof}
By the triangle inequality we have
$$\left\|\sum_{j\geq s}W_j^+(\mathcal F_{j-s})\right\|_{L^1(w)}\leq \sum_j
\sum_{I\in \mathcal{I}_{j-s}}
\int |f_I(y)|\left(\int |K_j(x,y)|w(x)\,dx\right)\,dy.
$$
We note that $K_j(x,y)$ is supported in the interval $[y-2^{j+1}, y-2^{j-1}]$
as a function of $x$, for each fixed $y$, and
$$\sup [y-2^{j+1}, y-2^{j-1}] \leq \inf I \qquad \text{for all $y\in I\in
\mathcal{I}_{j-s}$.} $$
Also, $|K_j| \leq C2^{-j}$.  Thus we have
$$\int |f_I(y)|\left(\int |K_j(x,y)|w(x)\,dx\right)\,dy
\leq C\int |f_I(y)|\inf_I M^-(w)\,dy \leq C\lambda |I|\inf_I w,  $$
where $M^-$ is as in (1.4).
Combining the results, we get the conclusion.
\end{proof}
\par
Let $\mathcal{J}$ denote the family of intervals arising from the
Calder\'{o}n-Zygmund decomposition in Lemma 3.1.
\begin{lemma}
Let $t>0$,  $w\in A_{1}^{+}$ and $s$ be a positive integer.
Let $\mathcal B_j, E_\alpha^s$ be as above.
Then we have
\begin{equation}
\int_{E_\lambda^s}\min(w(x),t) \,dx
\leq C \sum_{J \in \mathcal J}|J| \min\left(t2^{-s},
\inf_{J} w\right).
\end{equation}
\end{lemma}
\begin{proof}
Let
$$
\mathcal{J}_{t}=\{J\in \mathcal{J}:\inf_{J}w(x)<t2^{-s}\}
$$
and $\mathcal{J}_{t}^{c}=\mathcal{J}\setminus\mathcal{J}_{t}$. For $j>0$, put
$$\mathcal{B}_{j}^{'}=\sum_{2^{j-1}<|J|\leq 2^{j},J\in \mathcal{J}_{t}}b_{J},\,\,\,\,\,\, \,\,\,\,\,\,\mathcal{B}_{j}^{''}=\sum_{2^{j-1}<|J|\leq 2^{j},J\in
\mathcal{J}_{t}^{c}}b_{J},$$
and
$$\mathcal{B}_{0}^{'}=\sum_{|J|\leq1,J\in \mathcal{J}_{t}}b_{J}, \,\,\,\,\,\, \,\,\,\,\,\,\mathcal{B}_{0}^{''}=\sum_{|J|\leq1,J\in \mathcal{J}_{t}^{c}}b_{J}.$$
Then
$\mathcal{B}_{j}=\mathcal{B}_{j}^{'}+\mathcal{B}_{j}^{''}$ for $j\geq 0$.
Define
\begin{gather*}
E'_\alpha=\left\{x\in \Bbb R : \left|\sum_{j\geq s}
W_j^+(\mathcal B'_{j-s})(x)\right|>\alpha\right\},
\\
E^{\prime\prime}_\alpha=\left\{x\in \Bbb R : \left|\sum_{j\geq s}
W_j^+(\mathcal B^{\prime\prime}_{j-s})(x)\right|>\alpha\right\},
\end{gather*}
for $\alpha>0$.
Then,  we have
$E_\lambda^s \subset E'_{\lambda/2}\cup E^{\prime\prime}_{\lambda/2}$,
and hence
\begin{align*}
\int_{E_\lambda^s} \min(w(x),t)\, dx
&\leq \int_{E'_{\lambda/2}}\min(w(x),t)\, dx
+\int_{E^{\prime\prime}_{\lambda/2}} \min(w(x),t)\, dx
\\
&\leq \int_{E'_{\lambda/2}} w(x)\, dx  +\int_{E^{\prime\prime}_{\lambda/2}}
t\, dx.
\end{align*}
By Lemma 3.3 and Lemma 3.4,  with
$\mathcal{G}_{j}=C_{1}\mathcal{B}_{j}^{''}$ and $\mathcal{F}_{j}=C_{2}\mathcal{B}_{j}^{'}$, via Chebyshev's inequality, we have
\begin{gather*}
\int_{E'_{\lambda/2}} w(x)\, dx  \leq C\sum_{J \in \mathcal J_t}|J| \inf_{J} w
= C\sum_{J \in \mathcal J_t}|J|
\min\left(t2^{-s},\inf_{J} w\right),
\\
 \int_{E^{\prime\prime}_{\lambda/2}}
t\, dx   \leq C t2^{-s}\sum_{J\in \mathcal J_t^c}|J|
=  C\sum_{J \in \mathcal J_t^c}|J| \min\left(t2^{-s},\inf_{J}w\right).
\end{gather*}
Combining these estimates, we conclude the proof of Lemma 3.5.
\end{proof}
\par
Now, we prove Lemma 3.2.  Since
$$
\int_{0}^{\infty}\min(N,t)t^{-1+\theta}\,dt/t =C_{\theta}N^{\theta},
$$
for $0<\theta<1$, $C_{\theta},\,\,N>0$. Multiplying both sides of (3.10) by
$t^{-1+\theta} (0<\theta<1)$, then integrating them on $(0,\infty)$ with respect to the measure $dt/t$, we get
\begin{align*}
\int_{E_\lambda^s} w(x)^{\theta} \,dx
&\leq C \sum_{J \in \mathcal J}|J|2^{-(1-\theta)s} \inf_{J} w^{\theta}
  \\
&\leq  C\lambda^{-1}2^{-(1-\theta)s}
\sum_{J \in \mathcal J}\inf_{J} w^{\theta} \int_J|f(x)| \,dx
\\
&\leq  C\lambda^{-1} 2^{-(1-\theta)s}\int |f(x)| w(x)^{\theta} \,dx.
\end{align*}
By Proposition 1.7, if $w\in A_{1}^{+}$, then  $w^{1+\delta}\in A_{1}^{+}$
for some $\delta>0$. Therefore, we complete the proof of Lemma 3.2 by
substituting $w^{1+\delta}$ for $w$ and putting $\theta=\frac{1}{1+\delta}$
in the above inequalities.

\begin{lemma}
Let $W^{+}_{j}$ be as in $(3.9)$.
Suppose $w\in A_{1}^{+}$. There exist $C,\delta>0$ such that
$$
\| W^{+}_{j}\|_{L^{2}(w)}\leq C2^{-j\delta}
$$
for all $j\geq 1$, where $\|\cdot\|_{L^{2}(w)}$ denotes the operator norm
on $L^2(w)$.
\end{lemma}
Before proving Lemma 3.6, we first give a lemma  obtained by
Ricci-Stein.
\begin{lemma}[\cite{RS}]
For
$j\geq 1$, if $k\neq l$, we have
$$
\|W_{j}^{+}\|_{L^{2}}\leq C_{k,l}2^{-\frac{j}{2}-\min(\frac{l}{k},\frac{k}{l})\frac{j}{2}}
$$
and if $k=l$,
$$
\|W_{j}^{+}\|_{L^{2}}\leq C_{k}2^{-j}j^{\frac{1}{2}}.
$$
\end{lemma}
To prove Lemma 3.6, we apply interpolation  with change of measures \cite{SW}.
For $j\geq 1$, since
$$
|W_{j}^{+}(f)|\leq C\int_{2^{j-1}+x}^{2^{j+1}+x}\frac{|f(y)|}{|x-y|}\,dy
\leq CM^{+}(f)(x),
$$
 Theorem 1.5 and Proposition 1.7 imply that
$\|W_{j}^{+}\|_{L^{2}(w)}\leq C$ for $w\in A_{1}^{+}$.
Consequently,
\begin{equation}
\|W_{j}^{+}\|_{L^{2}(w^{1+\varepsilon})}\leq C,
\end{equation}
for some $\varepsilon>0$ for which  $w^{1+\varepsilon}\in A_{1}^{+}$ (see
Proposition 1.7). So, Lemma 3.6 follows from Lemma 3.7 and (3.11) by
interpolation  with change of measures.
\par
Lemma 3.2 and Lemma 3.6 are essential to the proof of Proposition 2.2.

\section{Proof of Proposition 2.2}

We first prove (2.2).
Take any $h\in \mathbb{R}$, and write
$$
P(x,y)=a_{kl}(x-h)^{k}(y-h)^{l}+R(x,y,h),
$$
where the polynomial $R(x,y,h)$ satisfies the induction assumption for
Theorem 2.1, and the coefficients of $R(x,y,h)$ depend on $h$.
Write
$$
T_{0}^{+}f(x)=T_{01}^{+}f(x)+T_{02}^{+}f(x),
$$
where
$$T_{01}^{+}f(x)=\mathrm{p.v.}\int_{x}^{1+x}e^{i\left(R(x,y,h)+a_{kl}(y-h)^{k+l}\right)}K(x-y)f(y)\,dy,$$
and
$$T_{02}^{+}f(x)=\mathrm{p.v.}\int_{x}^{1+x}\left\{e^{iP(x,y)}-e^{i(R(x,y,h)+a_{kl}(y-h)^{k+l})}\right\}K(x-y)f(y)\,dy.$$
\par
Now we split $f$ into three parts as follows:
$$
f(y)=f(y)\chi_{\{|y-h|<\frac{1}{2}\}}(y)+f(y)\chi_{\{\frac{1}{2}\leq|y-h|<\frac{5}{4}\}}(y)+f(y)\chi_{\{|y-h|\geq\frac{5}{4}\}}(y)=f_{1}(y)+f_{2}(y)+f_{3}(y).
$$
It is easy to see that  $|x-h|<\frac{1}{4}$ and $|y-h|<\frac{1}{2}$ imply
$|y-x|<1$,  and hence we have
$$
T_{01}^{+}f_1(x)=\mathrm{p.v.}\int e^{i(R(x,y,h)+a_{kl}(y-h)^{k+l})}K(x-y)
f_1(y)\,dy.
$$
Thus, from the induction assumption, it follows that
$$
w\left(\left\{x\in I(h,\frac{1}{4}):|T_{01}^{+}f_{1}(x)|>\lambda\right\}\right)\leq \frac{C}{\lambda}\int_{|y-h|<\frac{1}{2}}|f(y)|w(y)\,dy.\eqno(4.1)
$$
where $C$ is independent of $h$ and the coefficients of $P(x,y)$.
Here and after, $I(x,r)$ denotes the interval $(x-r,x+r)$.
\par
Notice that if $|x-h|<\frac{1}{4},\frac{1}{2}\leq|y-h|<\frac{5}{4}$, then
$|y-x|>\frac{1}{4}$. Thus
$$
|T_{01}^{+}f_{2}(x)|\leq \int_{x+\frac{1}{4}}^{x+1}|K(x-y)f_{2}(y)|\,dy
\leq CM^{+}(f_{2})(x).
$$
So we have
$$
w\left(\left\{x\in I(h,\frac{1}{4}):|T_{01}^{+}f_{2}(x)|>\lambda\right\}\right)\leq\frac{C}{\lambda}\int_{|y-h|<\frac{5}{4}}|f(y)|w(y)\,dy      \eqno(4.2)
$$
for some constant $C$ independent of $h$ and the coefficients of $P(x,y)$.
\par
Finally, if $|x-h|<\frac{1}{4},|y-h|\geq\frac{5}{4}$, then $|y-x|>1$, thus
$$
T_{01}^{+}f_{3}(x)=0.\eqno(4.3)
$$

From (4.1), (4.2) and (4.3), it follows that
$$
w\left(\left\{x\in I(h,\frac{1}{4}):|T_{01}^{+}f(x)|>\lambda\right\}\right)\leq\frac{C}{\lambda}\int_{|y-h|<\frac{5}{4}}|f(y)|w(y)\,dy,\eqno(4.4)
$$
where $C$ is independent of $h$ and the coefficients of $P(x,y)$.

Evidently, if $|x-h|<\frac{1}{4},0<y-x<1$, then
$$
\left|e^{iP(x,y)}-e^{i(R(x,y,h)+a_{kl}(y-h)^{k+l})}\right|\leq C|a_{kl}||x-y|
=C(y-x).
$$
Therefore, when $|x-h|<\frac{1}{4}$, we have
$$
|T_{02}^{+}f(x)|\leq C\int_{x}^{x+1}|f(y)|\,dy\leq CM^{+}(f(\cdot)\chi_{B(h,\frac{5}{4})}(\cdot))(x).
$$
It follows that
$$
w\left(\left\{x\in I(h,\frac{1}{4}):|T_{02}^{+}f(x)|>\lambda\right\}\right)\leq\frac{C}{\lambda}\int_{|y-h|<\frac{5}{4}}|f(y)|w(y)\,dy         \eqno(4.5)
$$
for some constant $C$ independent of $h$ and the coefficients of $P(x,y)$. From (4.4) and (4.5), it follows that the inequality
$$
w\left(\left\{x\in I(h,\frac{1}{4}):|T_{0}^{+}f(x)|>\lambda\right\}\right)\leq\frac{C}{\lambda}\int_{|y-h|<\frac{5}{4}}|f(y)|w(y)\,dy 
$$
holds uniformly in $h\in \mathbb{R}$, which implies
$$
w\left(\left\{x\in \mathbb{R}:|T_{0}^{+}f(x)|>\lambda\right\}\right)\leq\frac{C}{\lambda}\|f\|_{L^{1}(w)} 
$$
by integration with respect to $h$,
where $C$ is independent of the coefficients of $P(x,y)$.
This completes the proof of (2.2).
\par
Now, we turn to the proof of (2.3).
Recall that $T_{\infty}^{+}=W_{0}^{+}+W^{+}$.
It is easy to see that
$$\|W^{+}_{0}(f)\|_{L^1(w)}\leq C\|f\|_{L^1(w)}   \eqno(4.6)  $$
 for $w\in A_{1}^{+}$, since
\begin{eqnarray*}
\int |W^{+}_{0}(f)(x)|w(x)\, dx &\leq&
\iint |K_0(x-y)|w(x)\, dx |f(y)|\, dy
\\
&\leq& C\int M^{-}w(y)|f(y)|\, dy
\leq C\int w(y)|f(y)|\,dy.
\end{eqnarray*}
So,  in the following, we  only consider $W^{+}$.

Now, we recall the decomposition $f=g+b$ and the set
$\mathcal{E}=\bigcup \tilde{I}$ in Section 3, and we see that
\begin{eqnarray*}
&&w\left(\left\{x\in \mathbb{R}\setminus\mathcal{E}:|W^{+}(f)(x)|>\lambda \right\}\right)
\\
&\leq& w\left(\left\{x\in \mathbb{R}\setminus \mathcal{E}:|W^{+}(g)(x)|>\frac{\lambda}{2} \right\}\right)+w\left(\left\{x\in \mathbb{R}\setminus \mathcal{E}:|W^{+}(b)(x)|>\frac{\lambda}{2} \right\}\right)
\\
&\leq& C\lambda^{-2}\|W^+(g)\|^{2}_{L^{2}(w)}+
w\left(\left\{x\in \Bbb R^n ; \left|\sum_{s\geq 1}\sum_{j\geq s}
W_j^+(\mathcal B_{j-s})(x)\right|>\lambda/2\right\}\right).
\end{eqnarray*}
Form Lemma 3.6 we easily see that $W^{+}$ is bounded on $L^{2}(w)$. So,
$\lambda^{-2}\|W^+(g)\|^{2}_{L^{2}(w)}$ is bounded by
$C\lambda^{-1}\|f\|_{L^{1}(w)}$ via Lemma 3.1 (3.4), (3.5).
Checking the constants appearing in the proof of Lemma 3.2 and replacing
$K$ by $c2^{\delta s}K$,   we have
$$w\left(E_{c_\delta 2^{-\delta s}\lambda}^s\right)\leq c\lambda^{-1}
2^{-\tau s}\|f\|_{L^1(w)} , $$
where $\delta$ and $\tau$ are positive constants depending on $w$,
and $c_\delta$ is a constant satisfying
$\sum_{s\geq 1}c_\delta 2^{-\delta s}=1/2$.
Thus, we have
$$w\left(\left\{x\in \Bbb R^n ; \left|\sum_{s\geq 1}\sum_{j\geq s}
W_j^+(\mathcal B_{j-s})(x)\right|>\lambda/2\right\}\right)
\leq \sum_{s\geq 1}w\left(E_{c_\delta 2^{-\delta s}\lambda}^s\right)
\leq C\lambda^{-1}\|f\|_{L^1(w)}.$$
Therefore, we have
$$
w\left(\left\{x\in \mathbb{R}\setminus\mathcal{E}:|W^{+}(f)(x)|>\lambda \right\}\right)\leq C\lambda^{-1}\|f\|_{L^1(w)}.                   \eqno(4.7)
$$
\par
On the other hand, by Lemma 3.1 (3.3) and the estimate $w(\tilde{I})\leq Cw(I)$, which is easily proved by the condition $w\in A_1^+$, we see that
$$
w(\mathcal{E})\leq C\lambda^{-1}\|f\|_{L^{1}(w)}.    \eqno(4.8)
$$
By (4.7) and (4.8) for $w\in A_{1}^{+}$, we get
$$
w\left(\left\{x\in \mathbb{R}:|W^{+}(f)(x)|>\lambda \right\}\right) \leq C\lambda^{-1}\|f\|_{L^{1}(w)}.\eqno(4.9)
$$
The results (4.6) and (4.9) imply
$$
w\left(\left\{x\in \mathbb{R}:|T^{+}_{\infty}(f)(x)|>\lambda \right\}\right) \leq C\lambda^{-1}\|f\|_{L^{1}(w)}
$$
for $w\in A_{1}^{+}$ with a constant $C$ independent of the coefficients of
$P(x,y)$, which completes the proof of (2.3).


\end{document}